\documentclass[11pt, reqno]{article}

\usepackage{authblk}

\usepackage[a4paper]{geometry}
\usepackage[utf8]{inputenc}
\usepackage{xcolor}
\usepackage{amsmath,amsfonts,amssymb,amsthm}
\usepackage{graphicx}

\usepackage{hyperref}

\usepackage{cleveref}
  \crefname{section}{Section}{Sections}
  \crefname{figure}{Figure}{Figures}
  \crefname{theorem}{Theorem}{Theorems}
  \crefname{lemma}{Lemma}{Lemmas}
  \crefname{proposition}{Proposition}{Propositions}  
  \crefname{corollary}{Corollary}{Corollaries}
  \crefname{definition}{Definition}{Definitions}
  \crefname{example}{Example}{Examples}
  \crefname{remark}{Remark}{Remarks}

\newtheorem{theorem}{Theorem}[section]
\newtheorem{lemma}[theorem]{Lemma}
\newtheorem{definition}[theorem]{Definition}
\newtheorem{proposition} [theorem] {Proposition}
\newtheorem{corollary}[theorem]{Corollary}
\newtheorem{remark}[theorem]{Remark}
\newtheorem{example}[theorem]{Example}
\newtheorem*{question*}{Question}

\title{A support theorem for SLE curves}
\author[1]{Huy Tran\thanks{tran@math.tu-berlin.de}}
\author[2]{Yizheng Yuan\thanks{yuan@math.tu-berlin.de}}

\affil[1]{TU Berlin, Germany}
\affil[2]{TU and HU Berlin, Germany}

\newcommand*{\defeq}{\mathrel{\mathop:}=}
\newcommand*{\eqdef}{=\mathrel{\mathop:}}

\renewcommand*{\Re}{\operatorname{Re}}
\renewcommand*{\Im}{\operatorname{Im}}
\newcommand*{\osc}{\operatorname{osc}}
\newcommand*{\hcap}{\operatorname{hcap}}
\renewcommand*{\fill}{\operatorname{fill}}
\renewcommand*{\P}{\mathbb{P}}
\renewcommand*{\H}{\mathbb{H}}
\newcommand*{\barH}{\overline{\mathbb{H}}}
\newcommand*{\cD}{\mathcal{D}}
\newcommand*{\SLEk}{SLE$_\kappa$}
\begin{document}

\maketitle
\abstract{For all $\kappa > 0$, we show that the support of SLE$_\kappa$ curves is the closure in the sup-norm of the set of Loewner curves driven by nice (e.g. smooth) functions. It follows that the support is the closure of the set of simple curves starting at $0$.}

%
%
%
\section{Introduction}
\subsection{Overview}
The support of a random variable $X$ in a  Polish space  is the set of points $x$ such that for any open neighborhood $V$ of $x$, we have $\P(X\in V)>0$. In this paper, the random variable $X$ will be a random process, namely the \SLEk{} trace, and our goal is to describe its support.

Characterising the support of random processes such as Brownian motion and diffusions is an important research problem for stochastic (partial) differential equations, where it was initiated by Stroock and Varadhan \cite{SV} when they studied a strong maximum principle of a PDE operator. In \cite{CF11} a support theorem was the key to a H\"ormander/Malliavin theory for rough differential equations. The description of a support is also an important step to study the invariant measure of stochastic equations (see e.g. \cite{TW, CF}). Other questions related to support theorems are large deviation estimates, or the continuity of solution maps of SDE and SPDE.



\SLEk{} is an important random planar curve that shares many analogies with Brownian motion and other random processes. \SLEk{} is proven and conjectured to be the scaling limits or interface of many discrete models arisen from statistical physics (e.g. \cite{LSW04, Smi01, SS05, SS09, Che+}). Instead of the Markov property, it satisfies a domain Markov property. Depending on the parameter $\kappa$, it has different regularities (similar to fractional Brownian motion). Moreover, \SLEk{} is defined through a family of deterministic ordinary differential equations called Loewner equations with the random input $\sqrt{\kappa} B$, where $B$ is the one dimensional Brownian motion.

Motivated by the rich study of Brownian motion and other processes and by the similarities of \SLEk{} to them, it is natural to ask for a support theorem for \SLEk{} curves. Before stating the main result, let us give an overview of \SLEk{}.

The Loewner map (defined in \cref{s:def}) associates to certain real-valued continuous functions $\lambda\in C([0,1],\mathbb{R})$ a continuous non-crossing Loewner curve $\gamma^\lambda\in C([0,1],\overline{\mathbb{H}})$. The curve is constructed from a family of Loewner equations. Not every function in $C([0,1],\mathbb{R}) $ corresponds to a curve; see an example in \cite[Section 5]{MR}. It is proved that the Loewner curve $\gamma^\lambda$ is defined if locally the $1/2$-H\"older constant of $\lambda$ is less than 4 (\cite{MR}, \cite{Lind05}). 

We call the Loewner map with the input $\sqrt{\kappa}B$ the Schramm-Loewner map. It is shown that this map is almost surely well-defined (\cite[for $\kappa\neq 8$]{RS}, \cite[for $\kappa=8$]{LSW04}), that is, a.s. it gives rise to a curve. These random curves are called \SLEk{} curves, and (abusing notation) denoted $\gamma^\kappa$ (instead of $\gamma^{\sqrt\kappa B})$.

The previous properties could remind us of stochastic differential equations (SDE). An SDE is also driven by a Brownian motion, and if one replaces the Brownian motion by smooth functions, then the SDE becomes an ODE and has a deterministic solution. Recall that the support of the solution to an SDE can be characterized by the solutions of the ODEs that arise by replacing the Brownian noise by Cameron-Martin paths (see e.g. \cite[Chapter 19]{FV}). One could guess that the support of SLE can be described in the analogous way. We show in this paper that this is indeed true.

The main difficulty in proving such statements is that the Schramm-Loewner map (or, in the SDE case, solution map) is not continuous, and even only almost surely defined. If it were, the support theorem for SLE would immediately follow from the well-known support theorem for Brownian motion. Note also that the \SLEk{} curve is not a diffusion process, even though the Loewner equation with Brownian motion as an input can be seen as an SDE. Hence the method of proving support theorems for diffusion processes does not apply directly to \SLEk{}.

Consider the set $\cD$ of all functions that have locally vanishing $1/2$-H\"older constant. See Section \ref{s:def} for the exact definition and properties of $\mathcal{D}$. Our main theorem is the following.
\begin{theorem}\label{t:main}
Fix $\kappa>0$.  The support of \SLEk{}, parametrized by half-plane capacity,  in the space $C([0,1],\overline{\mathbb{H}})$ is the closure of $\{\gamma^\lambda:\lambda\in \cD\}$  with respect to the sup-norm topology.
\end{theorem}

Obviously $\cD$ contains $W^{1,2}$, the space of Cameron-Martin paths, so we can also describe the support of \SLEk{} as
\begin{equation*}
    \mathcal{S} = \overline{\{\gamma^\lambda \mid \lambda \in W^{1,2}, \lambda(0)=0\}},
\end{equation*}
in analogy to the corresponding result for SDE.

Moreover, the same set can be represented as
\begin{equation*}
    \mathcal{S} = \overline{\{\gamma \in C([0,1];\mathbb{H}) \mid \gamma \text{ simple, param. by half-plane capacity, and } \gamma(0)=0\}}.
\end{equation*}
See \cref{s:other_variants} for more details.

Note that in the statement of \cref{t:main} it is important to specify the topological space where the random process belongs to. There might be several ``natural'' spaces to which the random process corresponds. \SLEk{} can be viewed as a subset of the plane, a continuous path, an $\alpha$-H\"older function (\cite{Lind08}, \cite{JVL}), a $p$-variation path, or an element of Besov spaces (\cite{FT}). When we consider \SLEk{} only as compact subsets and measure distances by the Hausdorff metric, one can show a corresponding version of Theorem \ref{t:main} by applying the method in \cite[Lemma 8.2]{BJVK}. For \cref{t:main} which is a characterization of the support of \SLEk{} in the sup-norm, one needs a non-trivial effort. We believe that a similar statement can be made for Hölder and $p$-variation spaces (see discussion in \cref{s:other_variants}). (Similarly, there are different versions of the support theorem for Brownian motion and SDE; see for example \cite{LQZ}, \cite{AGL}.)


A direct consequence of \cref{t:main} is that all the above statements are true in the strong topology of curves, which is weaker than the sup-norm topology. Consider the space of continuous paths $\alpha:[0,1]\to\overline{\mathbb{H}}$ modulo reparametrisation. Then the strong topology is defined by the metric
\begin{equation*}
    \rho(\alpha,\beta) = \inf_\psi \sup_{t \in [0,1]} |\alpha(t)-\beta\circ\psi(t)|
\end{equation*}
where the infimum is taken over all increasing homeomorphisms from $[0,1]$ to $[0,1]$.

\begin{corollary}
    Fix $\kappa>0$.  The support of \SLEk{}, in the strong topology, is the closure of
    \begin{equation*}
        \{\gamma^\lambda:\lambda\in \cD\},
    \end{equation*}
    which is equal to the closure of
    \begin{equation*}
        \{\gamma \in C([0,1];\mathbb{H}) \mid \gamma \text{ simple}, \gamma(0)=0, \hcap(\gamma[0,1])=2\}.
    \end{equation*}
\end{corollary}


Theorem \ref{t:main} consists of the two following results.
\begin{proposition}\label{p:WZ}
	Let $\kappa \ge 0$. For each $\varepsilon>0$, almost surely there exists $\lambda\in \cD$ such that $
	\|\gamma^{\kappa}- \gamma^\lambda\|_{\infty,[0,1]} < \varepsilon$.
\end{proposition}

\begin{proposition}\label{p:main}
	Let $\kappa > 0$. For each $\varepsilon>0$ and $\lambda\in \cD$, we have $\P(\|\gamma^\kappa-\gamma^\lambda\|_{\infty,[0,1]}<\varepsilon)>0$.
\end{proposition}

The first proposition implies that the set $\overline{\{\gamma^\lambda:\lambda\in \cD\}}$ contains the support of \SLEk{}, while the second implies the other inclusion. Proposition \ref{p:WZ} is similar to the Wong-Zakai Theorem  \cite{WZ1, WZ2}, which is usually considered as an easier direction of support theorem. In principle, the Wong-Zakai Theorem says that if one regularizes or approximates the input (which is Brownian motion), then the output is also approximated. For \SLEk{}, $\kappa \neq 8$, this has been shown in \cite{Tran}. The result for $\kappa=8$ follows from \cite{LSW04}.

\begin{proof}[Proof of \cref{p:WZ}]
Let $\kappa \neq 8$. Fix a sample of $\xi = \sqrt{\kappa} B(\omega)$. Let $0=t_0 <t_1<\cdots < t_n=1$ with $t_k = \frac{k}{n}$ being a partition of $[0,1]$. Define $\lambda\in C([0,1], \mathbb{R})$ such that
\begin{itemize}
\item $\lambda(t_k) = \xi(t_k)$ for all $k$.
\item $\lambda$ is linear on $[t_k,t_{k+1}]$, i.e. \begin{equation*} \lambda(t) = \lambda(t_k) + n(\lambda(t_{k+1}))-\lambda(t_k))(t-t_k).\end{equation*}
\end{itemize}

The result in \cite[Section 4.1]{Tran} states that a.s. $\lim_{n\to \infty} \|\gamma^\lambda-\gamma^\xi\|_{\infty,[0,1]} = 0$. This shows \cref{p:WZ} in the case $\kappa\neq 8$ since the local $1/2$-H\"older constant of $\lambda$ is as small as we want, i.e. $\lambda\in \mathcal{D}$.

For $\kappa=8$, let $\gamma^8$ be a sample of the SLE$_8$ trace on the time interval $[0,1]$. Then it is almost surely in the support of SLE$_8$, i.e. for any $\varepsilon$-neighborhood $B_\varepsilon$ of $\gamma^8$, SLE$_8$ is in $B_\varepsilon$ with positive probability. From \cite[Theorem 4.8]{LSW04} it follows that with positive probability, a segment of some UST Peano curve $\hat\gamma$, mapped into $\mathbb{H}$, is also in $B_\varepsilon$. In particular, there is some sample of $\hat\gamma$ such that $\|\gamma^8-\hat\gamma\|_{\infty,[0,1]} < \varepsilon$.

Moreover, the UST Peano curve is constructed in \cite{LSW04} as a simple, piecewise smooth curve. Therefore, by rounding off the edges and reparametrising by half-plane capacity (a precise argument is conducted in the proof of \cref{p:simple_curves}), we find a smooth Loewner curve $\gamma$ (which in particular has a smooth driving function) with $\|\gamma^8-\gamma\|_{\infty,[0,1]} < \varepsilon$.
\end{proof}

Proving \cref{p:main} is the main part of this paper.

\subsection{Strategy}

First let us note a main difficulty in proving \cref{p:main}. Recall that in general, the Loewner map is not continuous, as the following example in \cite[Page 116]{Law} shows.

\begin{example}
 Let $\gamma^{(n)}$ be the simple polygonal path connecting the points $0$, $z_1$, $w_1$, $\hat z_1$, $\hat w_1$, $z_2$, $w_2$, $\hat z_2$, $\hat w_2$, ..., and parametrized by half-plane capacity, where
	\begin{align*}
	z_k = -\frac{1}{n}+i\frac{k}{n}, \quad w_k = i\frac{k}{2n}, \quad \hat z_k = \frac{1}{n}+i\frac{k}{n}, \quad \hat w_k = i\frac{k+1/2}{2n}.
	\end{align*}
	One can show that the driving function satisfies $|U^{(n)}_t| \le \frac{c}{\sqrt n}$ for $t \in [0,1]$ and some constant $c$. But the sequence $(\gamma^{(n)})_{n \in \mathbb{N}}$ has no convergent subsequence.
	
	Note that as sets, the traces $\gamma^{(n)}$ indeed come closer to the trace of the zero function, i.e. $\gamma(t) = i2\sqrt t$, but not as parametrized paths.
	
	\begin{figure}[ht]
		\centering
		\includegraphics[width=0.48\textwidth]{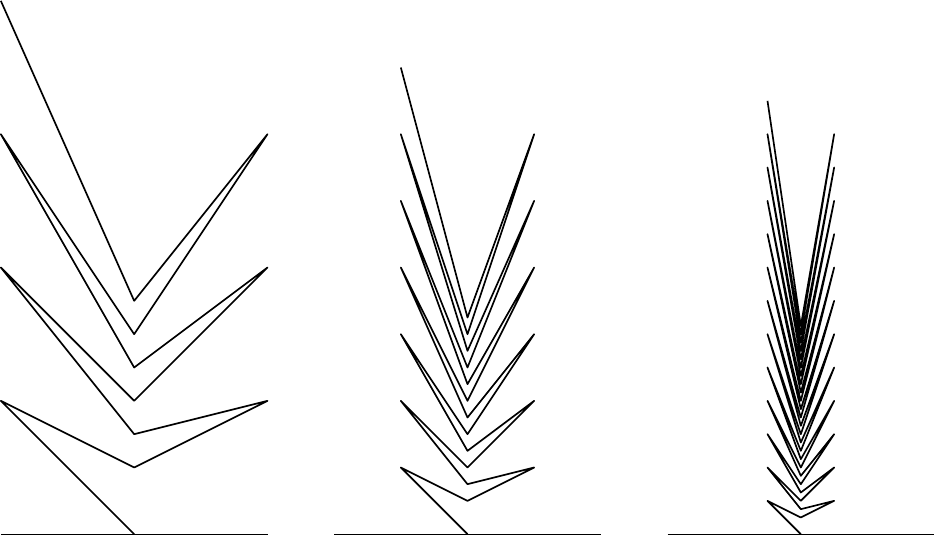}
		\caption{The ``Christmas tree''.}
		\label{fig:christmas_tree}
	\end{figure}

\end{example}

The above example shows that  with a small change  to a nice function $\lambda$, the curve $\gamma^\lambda$ could wiggle drastically.  Hence the event $\{\|\sqrt{\kappa} B - \lambda\|_{\infty,[0,1]}<\delta\}$ does not directly imply that $\|\gamma^\kappa - \gamma^\lambda\|<\epsilon$. 

The proof of \cref{p:main} will be of the form:
\begin{align}\label{e:form}
\text{If } (\xi,\gamma^\xi) \text{ satisfies } (A) \text{, then } \|\gamma^\xi-\gamma^\lambda\|_{\infty,[0,1]} \leq\varepsilon.
\end{align}
Naturally, one expects $(A)$ to contain the condition that $\|\xi-\lambda\|_\infty$ is small. But we need also something else that prevents a ``Christmas tree'' behaviour. The exact form of $(A)$ will be formulated in \cref{t:repeated_intervals}. The condition is roughly as follows:
\begin{align*}
\text{For all } k, \quad \xi \text{ is close to } \lambda \mbox{ on }  [t_k,t_{k+1}],
\end{align*}
where $0=t_0<t_1<\cdots< t_n=1$ is a partition of $[0,1]$ depending on $\lambda$ and $\varepsilon$, and the closeness of $\xi$ to $\lambda$ on $[t_k,t_{k+1}]$ depends on how $\gamma^\xi_{t_k}$ behaves on $[t_{k+1}, 1]$. ($\gamma^\xi_{t_k}$ denotes the trace of the Loewner chain driven by $\xi$ restricted to $[t_k, 1]$.) This structure of $(A)$ allows us to make use of the independent increments of Brownian motion, which will imply that $(A)$ is satisfied with positive probability.

Now we explain roughly how we estimate the difference $\|\gamma^\xi-\gamma^\lambda\|_{\infty,[0,1]}$ in \eqref{e:form}, and derive condition $(A)$. 

Let $0 \le t_0 < t_1 < t_2$ and $t \in [t_1,t_2]$. As in 
\cite{JVRW, Tran}, one uses either
\begin{multline}\label{e:diff}
|\gamma^\xi(t)-\gamma^\lambda(t)| \le |f^\xi_{t_0}(\gamma^\xi_{t_0}(t)+\xi(t_0))-f^\lambda_{t_0}(\gamma^\xi_{t_0}(t)+\xi(t_0))| \\
+ |f^\lambda_{t_0}(\gamma^\xi_{t_0}(t)+\xi(t_0))-f^\lambda_{t_0}(\gamma^\lambda_{t_0}(t)+\lambda(t_0))|
\end{multline}
or
\begin{multline}\label{e:diff1}
|\gamma^\xi(t)-\gamma^\lambda(t)| \le |f^\xi_{t_0}(\gamma^\xi_{t_0}(t)+\xi(t_0))-f^\xi_{t_0}(\gamma^\lambda_{t_0}(t)+\lambda(t_0))| \\
+ |f^\xi_{t_0}(\gamma^
\lambda_{t_0}(t)+\lambda(t_0))-f^\lambda_{t_0}(\gamma^\lambda_{t_0}(t)+\lambda(t_0))|
\end{multline}
where $f_{t_0}^\xi$ and $f_{t_0}^\lambda$ are reverse Loewner flows (see Section \ref{s:def}), and $\gamma^\xi_{t_0}$ and $\gamma^\lambda_{t_0}$ are the Loewner traces of $\xi$ and $\lambda$ started at time $t_0$ (see Section \ref{s:compare} for definitions). Which inequality should one use? 

The right-hand sides of (\ref{e:diff}) and (\ref{e:diff1}) have two things in common. They contain two terms. One is the difference between two conformal maps evaluated at the same point. The other term is the difference between the images of two points under the same map.

The first term of (\ref{e:diff1}) can be estimated since the expectation of the moments of $(f^\xi_{t_0})'$ have been studied carefully; see \cite{JVL}. This is the strategy used in \cite{Tran}. However, upon investigating, one  needs the expected moments of $(f^\xi_{t_0})'$ \emph{conditioned} on $\xi$ (which is a multiple of Brownian motion) close to a given $\lambda$, which is not known.

It turns out that the inequality (\ref{e:diff}) is approachable. To estimate the second term in the right-hand side of (\ref{e:diff}), we want the map $f^\lambda_{t_0}$ to be uniformly continuous, uniformly in $t_0$. This is true for sufficiently nice $\lambda$. This is where we impose the condition for $\lambda$ in Theorem \ref{t:main}.

To control the distance between $\gamma^\xi_{t_0}(t)$ and $\gamma^\lambda_{t_0}(t)$, we just need to observe that when $|t_2-t_0|$ is small, both points stay within a small box around $0$; see Lemma \ref{t:hull}.

For the first term of (\ref{e:diff}), we will apply (\ref{e:f1-f2}) of Lemma \ref{t:f1-f2}. This lemma concerns the difference between two conformal maps driven by two driving functions. Roughly speaking, it tells us that
\begin{equation*}
    |f^\xi_{t_0}(\gamma^\xi_{t_0}(t)+\xi(t_0))-f^\lambda_{t_0}(\gamma^\xi_{t_0}(t)+\xi(t_0))| \lesssim  \frac{\|\xi-\lambda\|_{\infty,[0,t_0]}}{\Im \gamma^\xi_{t_0}(t)}
\end{equation*}
where $a \lesssim b$ means $a \le Cb$ for some fixed constant $C>0$.

We get an estimate that can go arbitrarily bad if  $\gamma^\xi_{t_0}(t)$ gets close to the real line.  Note that $\gamma^\xi_{t_0}$ depends only on the increments of $\xi$ from $t_0$ onwards. Since Brownian increments on disjoint time intervals are independent, we can ``safely'' require a smaller value for $\|\xi-\lambda\|_{\infty,[0,t_0]}$, depending on $\inf_{t\in [t_1,t_2]} \Im \gamma^\xi_{t_0}(t)$. 
 
The aforementioned argument works for \SLEk{} with $\kappa\leq 4$ since a.s. \linebreak$\inf_{t\in [t_1,t_2]}\Im \gamma^\xi_{t_0}>0$ given fixed $t_0<t_1<t_2$. The situation becomes more complicated when $\kappa>4$ since 
\begin{equation*}
    \P(\inf_{t\in [t_1,t_2]}\Im \gamma^\xi_{t_0}=0)>0.
\end{equation*}

At the end, we show that it will not happen provided that $\xi$ is close to $\lambda$, i.e.
\begin{equation*}
    \P(\inf_{t\in [t_1,t_2]}\Im \gamma^\xi_{t_0}>0 \mid \xi \text{ close to } \lambda \text{ on } [t_0, t_2])=1.
\end{equation*}
This is another place where we will use the properties of functions in $\mathcal{D}$.

\subsection{Organization of the paper}
In \cref{s:def}, we gather some basic definitions and facts. In \cref{s:compare}, we prove a lemma comparing two deterministic Loewner curves. Then we use it in Section \ref{s:k_le_4} to prove Proposition \ref{p:main} in the case $\kappa\leq 4$. In Section \ref{s:k_gen}, we prove a lemma that generalizes the proof of Proposition \ref{p:main} to all $\kappa > 0$. In \cref{s:other_variants}, we discuss further characterisations of the support and some open questions.

\subsection{Acknowledgement}

We would like to thank Steffen Rohde and Peter Friz for various discussions. The authors acknowledge the financial support from the European Research Council (ERC) through a Consolidator Grant \#683164 (PI: Peter Friz).  We also thank Peter Friz and Vlad Margarint for suggestions regarding the presentation of the paper. We thank the referees for helpful comments and inspiring questions.


\section{Definitions and properties}\label{s:def}

\textbf{Definition of the Loewner map.} Let $\lambda\in C([0,1],\mathbb{R}) $. Consider the family of Loewner equations with different initial values:
\begin{alignat*}{2}
\partial_t g_t(z) &= \frac{2}{g_t(z)-\lambda(t)}, &\quad t\geq 0,\\
g_0(z) &= z, & z \in \mathbb{H}.
\end{alignat*}

For each $z\in \mathbb{H}$, there exists $T_z\in (0,\infty]$ where the equation has solution up to time $T_z$ at which $\lim_{t\to T_z^-} |g_t(z)-\lambda(t)|=0$. Define $K_t = \{z\in \overline{\mathbb{H}}: T_z\leq t\}$ for each $t\geq 0$. We call $K_t$ a compact $\H$-hull. One can show that $g_t$ is a conformal map from $\mathbb{H}\backslash K_t$ onto $\mathbb{H}$.

The following lemma concerns how big the hull $K_t$ is. See \cite[Lemma 3.2]{Wong} for a (trivial) proof.
\begin{lemma}\label{t:hull}
	Let $(K_t)$ be the hulls generated by a driving function $\lambda$. Then for all $z\in K_t$,
	\begin{align*}
	    |\Re z| \leq \sup_{s\in [0,t]} |\lambda(s)| \quad \text{and} \quad \Im z \leq 2 \sqrt{t}.
	\end{align*}
\end{lemma}

 If for every $t\geq 0$, the limit
\begin{align*}
    \gamma(t) \defeq \lim_{\mathbb{H} \ni z\to 0} g_t^{-1}(z+\lambda(t))
\end{align*}
exists and is continuous in $t\in [0,1]$, then $\mathbb{H}\backslash K_t$ is the unbounded component in $\mathbb{H}$ of $\mathbb{H}\backslash \gamma[0,t]$. The curve $\gamma\subset \overline{\mathbb{H}}$ is called the Loewner curve driven by $\lambda$. We call a Loewner curve simple if it intersects neither itself nor $\mathbb{R} \setminus \{\gamma(0)\}$. In that case, $K_t = \gamma[0,t]$ for all $t$.

For each $t\geq 0$, let
\begin{equation*}
    f_t = g_t^{-1} \quad \text{and} \quad \hat f_t=f_t(\cdot + \lambda(t)).
\end{equation*}
The maps $f_t$ and $\hat f_t$ are conformal on the upper half-plane $\mathbb{H}$. The latter is a centred version of $f_t$. To emphasize the dependence on $\lambda$, we also use notations $\gamma^\lambda$, $f_t^\lambda$, and the likes.

Let us denote $\Omega\subset C([0,1], \mathbb{R})$ the set of $\lambda$ that give rise to a curve. 

We call the map $\lambda \mapsto \gamma^\lambda$ from $\Omega$ to $C([0,1], \overline{\mathbb{H}})$ the Loewner map. It is known that:
\begin{itemize}
	\item $\Omega$ is not a convex space. Moreover, $\lambda\in \Omega$ does not imply $a\lambda\in \Omega$ for $a>0$. See \cite{LMR}.
	
	\item It follows from \cite{Lind05, LMR} that if $\lambda$ has local 1/2-H\"older norm less than 4, then $\lambda\in \Omega$. In particular, $C^\infty([0,1]) \subset W^{1,2}([0,1]) \subset \Omega$. 
	
	\item Let $\P$ be the Wiener measure on $C([0,1],\mathbb{R})$. For each $\kappa\geq 0$, $\P(\{\lambda: \sqrt{\kappa} \lambda \in \Omega\})=1.$
	
	\item We do not know whether $\P(\{\lambda: \sqrt{\kappa}\lambda\in  \Omega, ~\forall \kappa\}) = 1.$
\end{itemize}

\medskip

\noindent\textbf{The space $\mathcal{D}$.} 

We say that $\lambda$ has local $1/2$-H\"older norm less than $M>0$ if there exists $\delta>0$ such that
\begin{equation*}
    \sup_{|s-t|<\delta} \frac{|\lambda(s)-\lambda(t)|}{\sqrt{|s-t|}}<M.
\end{equation*}

\begin{definition} \label{d:D}
We say that $\lambda\in \mathcal{D}$ if $\lambda(0)=0$ and it has  locally vanishing $1/2$-H\"older norm, that is
\begin{equation*}
    \lim_{\delta\to 0}\sup_{|s-t|<\delta} \frac{|\lambda(s)-\lambda(t)|}{\sqrt{|s-t|}} = 0.
\end{equation*}
\end{definition}

\begin{proposition}
    Let $\lambda\in \mathcal{D}$. Then 
    \begin{itemize}
        \item $\lambda$ generates a simple curve.
        \item There is a function $\delta(\cdot,\lambda):(0,\infty)\to (0,\infty)$ such that 
            \begin{equation}\label{e:assum}
            |z_1-z_2| \le \delta(\varepsilon;  \lambda) \implies | f^\lambda_t(z_1)- f^\lambda_t(z_2)| \le \varepsilon \quad \forall t \in [0,1].
            \end{equation}
    \end{itemize}
\end{proposition}

A proof of this proposition can be found in the proof of \cite[Theorem 4.1]{LMR}. There they have shown that $\lambda$ generates a curve, and $\mathbb{H} \setminus \gamma^\lambda$ is a quasi-slit half-plane, therefore a John domain (see \cite[Section 5.2]{Pom} for a definition). It follows from \cite[Corollary 5.3]{Pom} that $f_t$ (as a conformal map from $\mathbb{H}$ to a John domain) is Hölder continuous on bounded sets, with Hölder constant and exponent depending on the John domain constant.

A  particular example: When $\lambda\equiv 0$, then $f^\lambda_{t}(z) = \sqrt{z^2-4t}$ which is H\"older-continuous in $z$ on any bounded set, uniformly in $t$. 

The following will be our main ingredient to estimate the difference between conformal maps derived from the Loewner equation.

\begin{lemma}[{\cite[Lemma 2.3]{JVRW}}]
	\label{t:f1-f2}
	Let $f^1_t$ and $f^2_t$ be two inverse Loewner maps with $U^1$ and $U^2$, respectively, as driving terms. Then for $t\geq 0$ and $z=x+iy \in \mathbb{H}$
	\begin{multline*}\label{e:f1-f2-full}
	|f^1_t(z)-f^2_t(z)| \\
	\le \|U^1-U^2\|_{\infty, [0,t]} \, \exp\left(\frac{1}{2}\left[\log\frac{I_{t,y}|(f^1_t)'(z)|}{y}\log\frac{I_{t,y}|(f^2_t)'(z)|}{y}\right]^{1/2}+\log\log\frac{I_{t,y}}{y}\right)
	\end{multline*}
	where $I_{t,y}=\sqrt{4t+y^2}$.
	
	Moreover,
	\begin{equation}\label{e:f1-f2}
	|f^1_t(z)-f^2_t(z)| \le \|U^1-U^2\|_{\infty, [0,t]} \left( \frac{I_{t,y}}{y}-1 \right).
	\end{equation}
\end{lemma}

\begin{remark}
	The inequality (\ref{e:f1-f2}) is the one that will be used. We do not use the full strength of Lemma \ref{t:f1-f2}. What we really need is an inequality of the form
	\begin{equation*}
	    |f^1_t(z)-f^2_t(z)| \le \Phi_1(\|U^1-U^2\|_{\infty, [0,t]}) \Phi_2(y)
	\end{equation*}
	where $\Phi_1,\Phi_2$ are two functions such that $\Phi_2(y)>0$ and $\Phi_1(0^+)=0$. Therefore, one can replace (\ref{e:f1-f2}) by an inequality in  \cite[Proposition 4.47]{Law} which says that
	\begin{equation*}
	    |f^1_t(z)-f^2_t(z)| \le \|U^1-U^2\|_{\infty, [0,t]} \left( e^{c_0t/y^2}-1 \right)
	\end{equation*}
	for some constant $c_0>0$.
	
\end{remark}


\section{Comparing two Loewner curves. A deterministic estimate.}
\label{s:compare}

We recall the following property of Loewner chains. For $\lambda \in C([0,1], \mathbb{R})$, we can run the Loewner chain from time $t_0 > 0$ instead of $0$, i.e. solve
\begin{align*}
    \partial_t g_{t_0,t}(z) &= \frac{2}{g_{t_0,t}(z)-(\lambda(t)-\lambda(t_0))}, \quad t \ge t_0,\\
    g_{t_0,t_0}(z) &= z.
\end{align*}
We will call the corresponding hulls $K_{t_0,t}$, and the trace (if it exists) $\gamma_{t_0}$.

If $\lambda$ generates a trace on $[0,t_0]$, and $\lambda(\cdot)-\lambda(t_0)$ generates a trace on $[t_0,1]$, then $\lambda$ generates a trace on $[0,1]$, and $\gamma(t) = \hat f_{t_0}(\gamma_{t_0}(t))$.

We will use the following lemma to compare the difference between two Loewner curves.
\begin{lemma}\label{t:fixed_interval} 
	Suppose $\lambda \in \mathcal{D}$ with the function $\delta(\cdot,\lambda)$ as in (\ref{e:assum}).
	Let $0\le t_0<t_1<t_2\le 1$ with $t_{k+1}-t_k\le\Delta t$, $k=0,1$.
	
	Let $\xi \in C([0,1],\mathbb{R})$ generate a Loewner trace $\gamma^\xi$, such that $\xi(t)-\xi(t_0)$, $t \in [t_0,t_2]$, generates a trace $\gamma^\xi_{t_0}$, and suppose 
	\begin{equation*}
	    c_{t_0,t_2} \defeq  \inf_{t \in [t_1,t_2]} \Im \gamma^\xi_{t_0}(t) > 0.
	\end{equation*}
	
	Moreover, suppose $\|\xi-\lambda\|_{[0,t_0]} \le \varepsilon_0$, $\|(\xi-\xi(t_0))-(\lambda-\lambda(t_0))\|_{[t_0,t_2]} \le \varepsilon_1$, and $\varepsilon_0, \varepsilon_1 \le \bar\varepsilon$ where $\varepsilon_0, \varepsilon_1, \bar\varepsilon > 0$.
	
	Then for any $a > 0$, $t\in[t_1,t_2]$ we have
	\begin{equation*}
	|\gamma^\xi(t)-\gamma^\lambda(t)| \le a + \varepsilon_0 c_{t_0,t_2}^{-1} \quad \text{if} \quad \varphi(2\Delta t; \lambda)+5\bar\varepsilon \le \delta(a; \lambda),
	\end{equation*}
	where $\varphi(\cdot; \lambda)$ is an increasing function with $\varphi(0^+;\lambda)=0$, depending on $\lambda$.
\end{lemma}

\begin{remark}
	The lemma roughly says that
	\begin{equation*}
	    \|\gamma^\xi-\gamma^\lambda\|_{[t_1,t_2]} \lesssim \Phi(\|\xi - \lambda\|_{[0,t_2]}) + \frac{\|\xi-\lambda\|_{[0,t_0]}}{c_{t_0,t_2}}
	\end{equation*}
	where $\Phi$ is an increasing function with $\Phi(0^+)=0$ that depends on the modulus of continuity of $\lambda$.
	
	Note that $c_{t_0,t_2}$ depends only on the increment $(\xi(t)-\xi(t_0))$, $t \in [t_0,t_2]$.  Therefore, if $\|\xi-\lambda\|_{[0,t_0]}$ is very small compared to $c_{t_0,t_2}$, then 	
	\begin{equation*}
	    \|\gamma^\xi-\gamma^\lambda\|_{[t_1,t_2]} \lesssim \|\xi - \lambda\|_{[t_0,t_2]}.
	\end{equation*}
	
	We also see that when $\gamma^\xi$ behaves like the ``Christmas tree'', then $c_{t_0,t_2}$ will be small. In order to prevent this behaviour, we can change $\xi$ on the interval $[0,t_0]$, making $\|\xi-\lambda\|_{[0,t_0]}$ smaller while leaving $c_{t_0,t_2}$ unchanged.
	
\end{remark}

\begin{proof}[Proof of Lemma \ref{t:fixed_interval}]
Let $\lambda$ and $\xi$ satisfy the conditions of \cref{t:fixed_interval}. Observe that $\|\xi-\lambda\|_{[t_0,t_2]} \le |\xi(t_0)-\lambda(t_0)|+\|(\xi-\xi(t_0))-(\lambda-\lambda(t_0))\|_{[t_0,t_2]} \le \varepsilon_0+\varepsilon_1$.
	
	Let $t\in[t_1,t_2]$. We follow (\ref{e:diff}) and estimate
	\begin{multline}\label{e:diffnew}
|\gamma^\xi(t)-\gamma^\lambda(t)| \le |f^\xi_{t_0}(\gamma^\xi_{t_0}(t)+\xi(t_0))-f^\lambda_{t_0}(\gamma^\xi_{t_0}(t)+\xi(t_0))| \\
+ |f^\lambda_{t_0}(\gamma^\xi_{t_0}(t)+\xi(t_0))-f^\lambda_{t_0}(\gamma^\lambda_{t_0}(t)+\lambda(t_0))|.
\end{multline}

First we estimate the second term of the right-hand side.
	
	Denoting the modulus of continuity of $\lambda$ by $\osc(\cdot;\lambda)$, we have $|\lambda(r)-\lambda(s)| \le \osc(|r-s|; \lambda)$, and consequently $|\xi(r)-\xi(s)| \le |\xi(r)-\lambda(r)|+|\lambda(r)-\lambda(s)|+|\lambda(s)-\xi(s)| \le \osc(|r-s|; \lambda)+4\bar\varepsilon$ for any $r,s \in [t_0,t_2]$. Therefore, by Lemma \ref{t:hull},
	\begin{align*}
	|\Re \gamma^\lambda_{t_0}(s)| &\le \osc(2\Delta t; \lambda),\\
	|\Im \gamma^\lambda_{t_0}(s)| &\le 2\sqrt{2\Delta t},\\
	|\Re \gamma^\xi_{t_0}(s)| &\le \osc(2\Delta t; \lambda)+4\bar\varepsilon,\\
	|\Im \gamma^\xi_{t_0}(s)| &\le 2\sqrt{2\Delta t},
	\end{align*}
	for $s \in [t_0,t_2]$.
This means that
\begin{equation}\label{e:phi}
|\gamma^\xi_{t_0}(t)+\xi(t_0)-\lambda(t_0)-\gamma^\lambda_{t_0}(t)| \le 2\osc(2\Delta t; \lambda)+5\bar\varepsilon+2\sqrt{2\Delta t} \eqdef \varphi(2\Delta t; \lambda
)+5\bar\varepsilon.
\end{equation}

Suppose $\varphi(2\Delta t; \lambda)+5\bar\varepsilon \le \delta(a; \lambda)$ for some $a > 0$. Then by the definition of $\delta$ and the above observation
\begin{align*}
|f^\lambda_{t_0}(\gamma^\xi_{t_0}(t)+\xi(t_0))-f^\lambda_{t_0}(\gamma^\lambda_{t_0}(t)+\lambda(t_0))| \le a
\end{align*}
which provides us a bound on the second term of (\ref{e:diffnew}).

For the first term of (\ref{e:diffnew}), we apply Lemma \ref{t:f1-f2} to obtain
\begin{align*}
|f^\xi_{t_0}(\gamma^\xi_{t_0}(t)+\xi(t_0))-f^\lambda_{t_0}(\gamma^\xi_{t_0}(t)+\xi(t_0))| \le \varepsilon_0 y^{-1}
\end{align*}
where $y=\Im\gamma^\xi_{t_0}(t) \ge c_{t_0,t_2}$.

Combining everything, we obtain
\begin{align*}
|\gamma^\xi(t)-\gamma^\lambda(t)| \le a + \varepsilon_0 c_{t_0,t_2}^{-1}
\end{align*}
for $t \in [t_1, t_2]$ if $\varphi(2\Delta t; \lambda)+5\bar\varepsilon \le \delta(a; \lambda)$.
\end{proof}

If we break $[0,1]$ into short sub-intervals, we can apply this argument on each sub-interval. On the very first sub-interval $t \in [0,t_1]$ we can directly estimate $|\gamma^\xi(t)-\gamma^\lambda(t)|$ with Lemma \ref{t:hull}. Together this will estimate $\|\gamma^\xi-\gamma^\lambda\|_{\infty,[0,1]}$. The precise conditions are the following.

\begin{corollary}\label{t:repeated_intervals}
	Suppose $\lambda \in \mathcal{D}$ with the function $\delta(\cdot,\lambda)$ as in (\ref{e:assum}).  	Let $0=t_0<t_1<...<t_n=1$ such that $t_k-t_{k-1}\le\Delta t$ for all $k \ge 1$. 
	
	Suppose $\xi \in C([0,1],\mathbb{R})$ with $\xi(0)=0$ generates a L\"owner trace such that
	\begin{equation*}
	    c_{t_k,t_{k+2}} \defeq \inf_{t \in [t_{k+1},t_{k+2}]} \Im \gamma^\xi_{t_k}(t) > 0 \quad \text{for all } k\ge 0,
	\end{equation*}
	where $\gamma^\xi_{t_k}$ is the L\"owner trace driven by $\xi(t_k+t)-\xi(t_k)$.
	
	Let $0<\bar\varepsilon < a$ be constants such that $\varphi(\Delta t; \lambda) < a$ and $\varphi(2\Delta t; \lambda)+5\bar\varepsilon \le \delta(a; \lambda)$, where $\varphi(\cdot; \lambda)$ is defined as in (\ref{e:phi}).
	
	Furthermore, suppose that $\varepsilon_1,...,\varepsilon_n$ are given such that $\varepsilon_k < \bar\varepsilon/2 \wedge a \, c_{t_k,t_{k+2}}$ and $\varepsilon_1 + ... + \varepsilon_k \le 2\varepsilon_k$ for all $k$, and moreover suppose that
	\begin{equation*}
	\|(\xi-\xi(t_{k-1}))-(\lambda-\lambda(t_{k-1}))\|_{\infty, [t_{k-1},t_k]} \le \varepsilon_k
	\end{equation*}
	for all $k \ge 1$.
	
	Then
	\begin{equation*}
	|\gamma^\xi(t)-\gamma^\lambda(t)| \le 3a
	\end{equation*}
	for all $t \in [0,1]$.
\end{corollary}

\begin{proof}
	Let $t \in [0,1]$. 	
	In case $t \le t_{1}$, applying Lemma \ref{t:hull} in the same way as in the proof of \cref{t:fixed_interval} implies
	\begin{equation*}
	|\gamma^\xi(t)-\gamma^\lambda(t)| \le \varphi(\Delta t; \lambda) + \bar\varepsilon < 2a.
	\end{equation*}
	
	If $t \ge t_{1}$, we find $k \ge 0$ such that $t \in [t_{k+1},t_{k+2}]$. We  apply Lemma \ref{t:fixed_interval} with the time points $0 \le t_k < t_{k+1} < t_{k+2}$.
	
	Observe that $\|\xi-\lambda\|_{[0,t_k]} \le \varepsilon_1+...+\varepsilon_k \le 2\varepsilon_k \le \bar\varepsilon$ and $\|(\xi-\xi(t_k))-(\lambda-\lambda(t_k))\|_{[t_k,t_{k+2}]} \le \varepsilon_{k+1}+\varepsilon_{k+2} \le 2\varepsilon_{k+2} \le \bar\varepsilon$.
	
	Lemma \ref{t:fixed_interval} shows
	\begin{equation*}
	|\gamma^\xi(t)-\gamma^\lambda(t)| \le a + 2\varepsilon_k c_{t_k,t_{k+2}}^{-1} < 3a.
	\end{equation*}
\end{proof}

\begin{remark}
    The list of conditions for \cref{t:repeated_intervals} looks quite long. We describe roughly how we will find suitable variables such that the corollary can be applied. 
    
    Suppose that $\lambda$ and $a$ are given. We will pick $\bar\varepsilon$ and $\Delta t$ accordingly. Then, to choose $\varepsilon_1,...,\varepsilon_n$ and $\xi$, note that each $c_{t_k,t_{k+2}}$ depends only on the increments of $\xi$ on the interval $[t_k,t_{k+2}]$. Therefore we can choose $\varepsilon_k$ depending on the increments of $\xi$ on $[t_k,1]$, and afterwards choose the increments of $\xi$ on $[t_{k-1},t_k]$, then again choose $\varepsilon_{k-1}$, and so on.
\end{remark}


\section{Proof of the Support Theorem for $\kappa \le 4$}\label{s:k_le_4}

Let $\xi(t) = \sqrt\kappa B_t$ and let $\lambda\in \mathcal{D}$. Let $a > 0$ be given.

For simplicity, we first show Theorem \ref{t:main} for $\kappa \le 4$. In this case $\gamma^\xi$ is a simple trace, as well as $\gamma^\xi_{t_k}$ for all $t_k \ge 0$. In particular, it will never touch the real line after time $0$ and automatically guarantees the condition $c_{t_k,t_{k+2}} = \inf_{t \in [t_{k+1},t_{k+2}]} \Im \gamma^\xi_{t_k}(t) > 0$ of Corollary \ref{t:repeated_intervals}.

The remaining task is to find a set of positive probability where all conditions of Corollary \ref{t:repeated_intervals} are satisfied.

First choose $\Delta t > 0$ and $\bar\varepsilon < a$ such that $\varphi(\Delta t; \lambda) < a$ and $\varphi(2\Delta t; \lambda)+5\bar\varepsilon \le \delta(a; \lambda)$. Next, partition $[0,1]$ into sub-intervals $0=t_0<t_1<...<t_n=1$ such that $|t_k-t_{k-1}|\le\Delta t$ for all $k \ge 1$.

Suppose now that we have (arbitrary) random variables $\varepsilon_k \le \bar\varepsilon$ that are a.s. positive and measurable w.r.t. $\mathcal F_{t_k,1}$ (where $\mathcal F_{r,s}$ denotes the sigma algebra generated by Brownian increments between time $r$ and $s$). By inductively applying  the independence of Brownian increments, we claim that
\begin{equation*}
\P(\forall k: \|(\xi-\xi(t_{k-1}))-(\lambda-\lambda(t_{k-1}))\|_{\infty, [t_{k-1},t_k]} \le \varepsilon_k)>0.
\end{equation*}

To verify this claim, suppose that for some $1\leq k< n$ 
\begin{equation*}
\P(\forall l\geq k+1: \|(\xi-\xi(t_{l-1}))-(\lambda-\lambda(t_{l-1}))\|_{\infty, [t_{l-1},t_l]} \le \varepsilon_l)>0.
\end{equation*}

Then since $\varepsilon_k$ is a.s. positive,
\begin{equation*}
\P(\varepsilon_k \ge b \quad \text{and} \quad \forall l \geq k+1: \|(\xi-\xi(t_{l-1}))-(\lambda-\lambda(t_{l-1}))\|_{\infty, [t_{l-1},t_l]} \le \varepsilon_l)>0
\end{equation*}
 for $b>0$ small enough. Since the Brownian increments on $[t_{k-1}, t_k]$ are independent of $\mathcal F_{t_k,1}$, it follows that
\begin{align*}
&\P(\|(\xi-\xi(t_{k-1}))-(\lambda-\lambda(t_{k-1}))\|_{\infty, [t_{k-1},t_k]} \le b\\
&\qquad \mid \varepsilon_k \ge b \ \text{and} \ \forall l\geq k+1: \|(\xi-\xi(t_{l-1}))-(\lambda-\lambda(t_{l-1}))\|_{\infty, [t_{l-1},t_l]} \le \varepsilon_l)\\
&= \P(\|(\xi-\xi(t_{k-1}))-(\lambda-\lambda(t_{k-1}))\|_{\infty, [t_{k-1},t_k]} \le b)\\
&> 0
\end{align*}
and consequently
\begin{equation*}
\P(\forall l \ge k: \|(\xi-\xi(t_{l-1}))-(\lambda-\lambda(t_{l-1}))\|_{\infty, [t_{l-1},t_l]} \le \varepsilon_l) > 0,
\end{equation*}
which implies the claim.

Now, it remains to choose suitable $\varepsilon_k$. Since $\kappa \le 4$, the curve $\gamma^\xi_{t_k}$ a.s. does not hit the real line for all $k$. Hence, the random variable $c_{t_k,t_{k+2}}(\omega) \defeq \inf_{t \in [t_{k+1},t_{k+2}]} \Im \gamma^{\xi(\omega)}_{t_k}(t)$ is a.s. positive. It is also measurable w.r.t. $\mathcal F_{t_k,t_{k+2}}$.\footnote{Note that the solution of the Loewner ODE is measurable with respect to the driver since it can be seen as the limit of a Picard iteration.}

Then, inductively backward in $k$, choose
\begin{equation}\label{e:e_k}
    \varepsilon_k = \frac{\varepsilon_{k+1}}{2} \wedge a \, c_{t_k,t_{k+2}}.
\end{equation}

Finally, every
\begin{equation*}
\omega \in \{ \forall k: \|(\xi-\xi(t_{k-1}))-(\lambda-\lambda(t_{k-1}))\|_{\infty, [t_{k-1},t_k]} \le \varepsilon_k \}
\end{equation*}
satisfies the conditions of Corollary \ref{t:repeated_intervals}, and therefore,
\begin{equation*}
\|\gamma^{\xi(\omega)}(t)-\gamma^\lambda(t)\|_{\infty,[0,1]} \le 4a.
\end{equation*}

This finishes the proof of \cref{p:main} in the case $\kappa\in (0,4]$.


\section{Proof of the Support Theorem for general $\kappa$}\label{s:k_gen}

In case $\kappa > 4$, we can use the same proof as before, but the condition 
\begin{equation*}
    c_{t_0,t_{2}}(\omega) = \inf_{t \in [t_{1},t_{2}]} \Im \gamma^{\xi(\omega)}_{t_0}(t) > 0
\end{equation*}
might be violated. Hence, our main task here is to add some condition that guarantees $c_{t_0,t_{2}}(\omega) > 0$ for almost all $\omega$.

Our main idea is as follows. Let $0 \le t_0 < t_1 < t_2$. We write $\gamma^\xi_{t_0}(t) = \hat f^\xi_{t_0,t_1}(\gamma^\xi_{t_1}(t))$ where $f^\xi_{t_0,t_1}$ is a conformal map that (extended to the boundary) maps some real interval $I$ onto $\partial K^\xi_{t_0,t_1}$,
\footnote{We consider hulls $K$ as closed subsets of the space $\barH$, so $\partial K = \{z\in K \mid B(z,\delta)\cap (\H \setminus K) \neq \varnothing \text{ for all } \delta>0\}$.}
and $\mathbb{R}\setminus I$ into $\mathbb{R} \setminus \{0\}$. If $\gamma^\xi_{t_1}(t) \in \mathbb{H}$, then trivially $\gamma^\xi_{t_0}(t) \in \mathbb{H}$. Hence, we only need to focus on the case  $\gamma^\xi_{t_1}(t) \in \mathbb{R}$.
Recall that if $\|(\xi-\xi(t_1))-(\lambda-\lambda(t_1))\|_{[t_1,t_2]} < \bar\varepsilon$, Lemma \ref{t:hull} implies
\begin{equation*}
|\Re \gamma^\xi_{t_1}(t)| \le \osc(\Delta t; \lambda)+2\bar\varepsilon \eqdef \alpha
\end{equation*}
for $t \in [t_1,t_2]$. We will show that, under some conditions,
\begin{equation}\label{e:goal_general_kappa}
    \text{the real interval } [-\alpha, \alpha] \text{ is contained in the interior of } I.
\end{equation}
This implies $\gamma^\xi_{t_1}(t) \in I$ and consequently $\gamma^\xi_{t_0}(t) \in \partial K^\xi_{t_0,t_1}$. Moreover, by a property of SLE proven by D. Zhan (see below), $\partial K^\xi_{t_0,t_1}$ intersects $\mathbb{R}$ only at its endpoints. Since $\gamma^\xi_{t_1}(t)$ actually lies in the interior of $I$, then $\gamma^\xi_{t_0}(t)$  lies in the interior of $\partial K^\xi_{t_0,t_1}$ which is contained in $\mathbb{H}$.

Finally, note that \eqref{e:goal_general_kappa} is reasonable because by the assumption $\lambda\in \mathcal{D}$ and the freedom to choose $\bar{\varepsilon}$, we can estimate
\begin{equation*}
    \alpha \approx (\text{a small number} ) \cdot \sqrt{\Delta t},
\end{equation*}
where for $I$, at least in the extreme case $\xi\equiv 0$, it can be calculated that $I=[-2\sqrt{\Delta t}, 2\sqrt{\Delta t}].$

Now, we fill the above arguments with more rigorous details. First,  we analyze $f^\xi_{t_0,t_1}$ as the inverse map of the Loewner flow driven by $\xi(\cdot)-\xi(t_0)$, $t \in [t_0,t_1]$.

In order to do so, we analyze the time-reversed Loewner equation. Recall that  if $(g_t)$ is the Loewner flow driven by $\xi$, then for any $s_0 > 0$ we can write $\hat f_{s_0} = h_{s_0}+\xi(s_0)$ where $(h_t)$ is the solution of
\begin{align*}
\partial_t h_t(z) = - \frac{2}{h_t(z)-W(t)}, \quad h_0(z) = z,
\end{align*}
with $W(t) = \xi(s_0-t)-\xi(s_0)$. For all $z \in \mathbb{C} \setminus \{W(0)\}$, this ODE can be solved on $t \in [0, T(z)[$ where $T(z)$ is the first time when $h_t(z)-W(t)$ hits $0$, and $T(z) = \infty$ for all $z \notin \mathbb{R}$.

Suppose that $\hat{f}_{s_0}:\mathbb{H}\to \mathbb{H}\backslash K_{s_0}$ can be continuously extended to the boundary $\mathbb{R}$. (This holds when the driver is in $\mathcal{D}$, or a multiple of Brownian motion.) It is known that exists a closed interval $I$ such that
\begin{equation}\label{e:I}
I= \{x\in \mathbb{R}: T(x)\leq s_0\} = \{x\in \mathbb{R}: \hat{f}_{s_0}(x)\in K_{s_0}\}.
\end{equation}

Therefore we can analyze the interval $I$ just by the time-reversed Loewner equation.

J. Lind has shown in \cite[Corollary 1]{Lind05} that if $W$ has $1/2$-H\"older constant less than $4$, then $T(x)$ is comparable to $x^2$. A comparison argument will show that the result stays true if the driver is slightly modified. The next two results make it more precise.

\begin{lemma}
	Let $V^1, V^2 \in C([0,\infty),\mathbb{R})$ with $V^1(0)=V^2(0)=0$ and let $(h^j_t)_{t \ge 0}$ solve
	\begin{align*}
	\partial_t h^j_t(x) = - \frac{2}{h^j_t(x)-V^j(t)}, \quad h^j_0(x) = x
	\end{align*}
	for $x \in \mathbb{R} \setminus \{0\}$, $t \in [0,T^j(x))$ where $T^j(x)$ is the first time when $h^j_t(x)-V^j(t)$ hits $0$, $j=1,2$.
	
	Let $t>0$, $x>0$, and $\delta = \|V^1-V^2\|_{\infty, [0,t]}$. If $T^1(x+\delta) \le t$, then $T^2(x) \le t$.
\end{lemma}

\begin{proof}
	Assume without loss of generality that $T^1(x+\delta) = t$, i.e. $h^1_s(x+\delta)$ exists for all $s<t$ and only dies at time $t$.
	
	We claim that for all $s<t$, $\partial_s h^2_s(x) < \partial_s h^1_s(x+\delta)$. (Note that this means $|\partial_s h^2_s(x)| > |\partial_s h^1_s(x+\delta)|$ since both are negative.)
	
	At $s=0$ this is obviously true since
	\begin{align*}
	- \frac{2}{h^2_0(x)-V^2(0)} = - \frac{2}{x} < - \frac{2}{x+\delta} = - \frac{2}{h^1_0(x+\delta)-V^1(0)}.
	\end{align*}
	
	Now if the claim holds for all $s<s_0$, then
	\begin{align*}
	h^2_{s_0}(x) - x = \int_0^{s_0} \partial_s h^2_s(x) \, ds < \int_0^{s_0} \partial_s h^1_s(x+\delta) \, ds = h^1_{s_0}(x+\delta) - (x+\delta).
	\end{align*}
	Consequently,
	\begin{align*}
	h^2_{s_0}(x)-V^2(s_0) \le h^2_{s_0}(x)-V^1(s_0)+\delta < h^1_{s_0}(x+\delta)-V^1(s_0),
	\end{align*}
	i.e.
	\begin{align*}
	- \frac{2}{h^2_{s_0}(x)-V^2(s_0)} < - \frac{2}{h^1_{s_0}(x+\delta)-V^1(s_0)}.
	\end{align*}
	This shows that there cannot be a first time $s_0$ where the claim is violated. By the continuity of $V^j$ and $h^j_t(x)$, and therefore also $\partial_t h^j_t(x)$ in $t$, the claim is never violated at any time.
	
	To finish the proof of the lemma, note that we have also shown above that
	\begin{align*}
	h^2_s(x)-V^2(s) < h^1_s(x+\delta)-V^1(s)
	\end{align*}
	for all $s \in [0,t]$. If $T^1(x+\delta) \le t$, this means that $h^1_s(x+\delta)-V^1(s) = 0$ for some $s \le t$, and consequently $h^2_s(x)-V^2(s) = 0$ for some smaller $s < t$.
\end{proof}

\begin{corollary}\label{t:interval_eaten}
	Let $V \in C^{1/2}([0,T],\mathbb{R})$ with $\|V\|_{1/2} < 4$ and $V(0)=0$. Then there exists some constant $c > 0$, depending on $\|V\|_{1/2}$, such that if $W \in C[0,T]$, $W(0)=0$, $\|W-V\|_\infty \le c \sqrt t$, and $|x| \le c \sqrt t$, then $T^W(x) \le t$. 
\end{corollary}

\begin{proof}
	By symmetry, it suffices to consider $x>0$. 
	By \cite[Corollary 1]{Lind05}, there exists a constant $c'$ depending on $\|V\|_{1/2}$ such that if $x \le c' \sqrt t$, then $T^V(x) \le t$.
	
	Let $c \defeq c'/2$. If $x \le c \sqrt t$, then $x+\|W-V\|_\infty \le c' \sqrt t$, so $T^V(x+\|W-V\|_\infty) \le t$. The previous lemma implies $T^W(x) \le t$.
\end{proof}

Now, we apply this corollary to our context of $\lambda\in \mathcal{D}$ and $\xi$ that is close to $\lambda$. 

\begin{corollary}\label{t:gamma_in_H}
	Let $\lambda\in \mathcal{D}$. Then for sufficiently small $\Delta t > 0$ (depending on $\lambda$) and $\bar\varepsilon > 0$ (depending on $\lambda$ and $\Delta t$) the following holds:
	
	Let $\xi \in C[0,1]$ with $\xi(0)=0$. Suppose that $\xi$ generates a Loewner trace $\gamma^\xi$, and $\xi(t)-\xi(\Delta t)$, $t \in [\Delta t, 2\Delta t]$, generates a trace $\gamma^\xi_{\Delta t}$, and $\|\xi-\lambda\|_{[0,\Delta t]} \le \bar\varepsilon$, $\|(\xi-\xi(\Delta t))-(\lambda-\lambda(\Delta t))\|_{[\Delta t, 2\Delta t]} \le \bar\varepsilon$.
	
	Then if $\gamma^\xi_{\Delta t}(t) \in \mathbb{R}$ for some $t \in [\Delta t, 2\Delta t]$, then $\gamma^\xi(t) \in K^\xi_{\Delta t}$.
	
	If additionally $\partial K^\xi_{\Delta t}$ intersects $\mathbb{R}$ only at its endpoints, then $\Im \gamma^\xi(t) > 0$.
\end{corollary}

In other words, the corollary states that $\inf_{t\in [\Delta t, 2\Delta t]} \Im(\gamma^\xi(t))>0$ when $\xi$ is close enough to $\lambda$ in a quantitative way.

\begin{proof}
	Since the constant $c>0$ in \cref{t:interval_eaten} depends only on $\|V\|_{1/2}$, we can fix $c>0$ corresponding to, say, $\|V\|_{1/2}=3$. Let $\Delta t$ be small enough such that on the interval $[\Delta t, 2\Delta t]$, the $1/2$-H\"older constant of $\lambda$ is less than $c/2$, and let $\bar\varepsilon \le \frac{c}{4}\sqrt{\Delta t}$. By Lemma \ref{t:hull},
	\begin{align*}
	|\Re \gamma^\xi_{\Delta t}(t)| \le \frac{c}{2} \sqrt{\Delta t} +2\bar\varepsilon \le c \sqrt{\Delta t}
	\end{align*}
	for $t \in [\Delta t, 2\Delta t]$.
	
	If now $\gamma^\xi_{\Delta t}(t) \in \mathbb{R}$, then \cref{t:interval_eaten}, applied to $V(s) = \lambda(\Delta t - s)-\lambda(\Delta t)$ and $W(s) = \xi(\Delta t - s)-\xi(\Delta t)$, and (\ref{e:I}) imply that
\begin{equation*}
	\gamma^\xi(t) = \hat f^\xi_{\Delta t}(\gamma^\xi_{\Delta t}(t)) \in \partial K^\xi_{\Delta t}.
\end{equation*}

Now suppose that $\partial K^\xi_{\Delta t}$ intersects $\mathbb{R}$ only at its endpoints. Then $\gamma^\xi(t) \in \mathbb{H}$ as long as it is not one of the endpoints.

In the above argument, $c$ and $\Delta t$ can be chosen such that the set $\{x \in \mathbb{R} \mid \hat f^\xi_{\Delta t}(x) \in K^\xi_{\Delta t}\}$ contains more than the interval $[-c\sqrt{\Delta t}, c\sqrt{\Delta t}]$. Then this interval gets mapped to an inner segment of $\partial K^\xi_{\Delta t}$, and not to its endpoints. In particular, $\gamma^\xi(t)$ is in the interior of $\partial K^\xi_{\Delta t}$.
\end{proof}

We remark that the assumption on $\xi$ holds almost surely if $\xi$ is a multiple of Brownian motion.

\begin{lemma}[{\cite[Theorem 6.1]{Zhan}}]\label{t:cut_points}
	Let $\kappa > 4$, and $(K_t)$ the hulls of \SLEk{}. For any $t>0$, almost surely $\partial K_t$ intersects $\mathbb{R}$ only at its endpoints.
\end{lemma}

Now, we can prove Proposition \ref{p:main}.
\begin{proof}[Proof of Proposition \ref{p:main}]
	The case $\kappa \le 4$ has already been shown in Section \ref{s:k_le_4}. The proof for $\kappa > 4$ is almost identical.
	
	Let $a>0$ be given. Choose $\Delta t > 0$ and $\bar\varepsilon < a$ such that $\varphi(\Delta t; \lambda) < a$ and $\varphi(2\Delta t; \lambda) + 5\bar\varepsilon \le \delta(a; \lambda)$. This time we additionally require $\Delta t$ and $\bar\varepsilon$ to satisfy the condition of Corollary \ref{t:gamma_in_H}.
	
	Then we partition the interval $[0,1]$ into $0=t_0 < t_1 < ...$ such that $|t_k-t_{k-1}| = \Delta t$, $k \ge 1$.
	
	The random variables $\varepsilon_k \le \bar\varepsilon$ are chosen as in (\ref{e:e_k}). \cref{t:gamma_in_H} together with \cref{t:cut_points}, applied to $\xi-\xi(t_k)$, $t \in [t_k, t_{k+2}]$, imply that a.s.
	\begin{equation*}
	    c_{t_k,t_{k+2}} = \inf_{t \in [t_{k+1},t_{k+2}]} \Im \gamma^\xi_{t_k}(t) > 0
	\end{equation*}
	for all $k \ge 0$. Hence, $\varepsilon_k$ are a.s. positive. The argument in Section \ref{s:k_le_4} shows that the event
	\begin{equation*}
	\forall k: \|(\xi-\xi(t_{k-1}))-(\lambda-\lambda(t_{k-1}))\|_{\infty, [t_{k-1},t_k]} \le \varepsilon_k
	\end{equation*}
	has positive probability.
	 
	For each $\omega$ in this event,  Corollary \ref{t:repeated_intervals} implies that $\|\gamma^{\xi(\omega)}-\gamma^\lambda\|_{\infty,[0,1]}\leq 4a$.  This finishes the proof.
\end{proof}


\begin{section}{Further characterisations of the support and open questions}
\label{s:other_variants}

We note that the set 
\begin{equation*}
    \mathcal{S} \defeq \overline{\{\gamma^\lambda:\lambda\in \mathcal{D}\}} \quad \subseteq C([0,1];\barH)
\end{equation*}
is a deterministic set and does not depend on $\kappa$. One may ask for what specific $\lambda$ (besides $\lambda \in \cD$) we have $\gamma^\lambda\in \mathcal{S}$?

First, it is worth mentioning that all curves in $\mathcal{S}$ are indeed Loewner curves, i.e. they satisfy the local growth property (which is not obvious since the closure is taken in the space $C([0,1];\overline{\mathbb{H}}$). We show this below in \cref{thm:limits_of_simple_curves}.

First recall that the half-plane capacity enjoys a uniform continuity property, described in \cite[Lemma 4.4]{Kem}. We will apply it in the following way.
\begin{lemma}\label{thm:hcap_continuity}
    For any $R > 0$ and $\varepsilon > 0$ there exists $\delta > 0$ such that if $K_1, K_2$ are compact $\mathbb{H}$-hulls with radius less than $R$ such that $K_1 \subseteq \fill(\overline{K_2^\delta})$ and vice versa, then $|\hcap(K_1)-\hcap(K_2)| < \varepsilon$.
\end{lemma}

Here, for a compact set $A \subseteq \barH$, $\fill(A)$ denotes the complement of the unbounded connected component of $\H \setminus A$, and $A^\delta$ the $\delta$-neighbourhood of $A$.

For a sequence of domains $H^n \subseteq \H$ that contain a common neighbourhood of $\infty$, their kernel (with respect to $\infty$) is the largest domain $H$ containing a neighbourhood of $\infty$ such that any compact $K \subseteq H$ is contained in all but finitely many $H^n$.

\begin{lemma}\label{thm:ball_components_limit}
Let $H^n \subseteq \H$ be a sequence of simply connected domains that contain a common neighbourhood of $\infty$, and let $H$ be their kernel (with respect to $\infty$). Let $z \in \H$, $0 < r_1 < r_2$, and $z_1,z_2 \in B(z,r_1) \cap H$. If for all $n$ the points $z_1$ and $z_2$ are in the same connected component of $B(z,r_1) \cap H^n$, then they are in the same connected component of $B(z,r_2) \cap H$.
\end{lemma}

\begin{proof}
Since $H$ is a domain, we can find a simple polygonal path $\alpha_1$ in $H$ from $z_1$ to $z_2$. Note that such a path hits $\partial B(z,r_1)$ only a finite number of times. Moreover, by a small perturbation we can choose $\alpha_1$ to cross $\partial B(z,r_1)$ at each such time. If $\alpha_1$ does not cross $\partial B(z,r_1)$ at all, we are done, so assume from now on that it does. Let $U \subseteq H$ be an open neighbourhood of $\alpha_1$. The definition of kernel implies $U \subseteq H^n$ for all but finitely many $n$. Without loss of generality we restrict ourselves to that subsequence.

Suppose for the moment that (small neighbourhoods of) $z_1$ and $z_2$ lie in the same connected component of $B(z,r_1) \setminus \alpha_1$. This means that $z_1$ and $z_2$ can be connected by a simple path $\alpha_2$ in $B(z,r_1)$ that does not intersect $\alpha_1$ except at its endpoints. In that case, $\alpha_1 \cup \alpha_2$ is a simple loop and by the Jordan curve theorem separates $\hat{\mathbb{C}}$ into two components. Call the component that contains $\infty$ the ``outside'' component.

By construction, $\alpha_1 \cup \alpha_2$ separates $\partial B(z,r_1)$ into finitely many segments, alternatingly ``inside'' and ``outside''. Let $A$ be an ``inside'' segment. Then there exists an open connected set $U_A \subseteq B(z,r_2) \setminus \overline{B(z,r_1)}$ in the neighbourhood of $A$ that is still ``inside'' $\alpha_1 \cup \alpha_2$. We claim that $U_A \subseteq H^n$ for all $n$. This will imply that $\hat U \defeq U \cup \bigcup\limits_{A\text{ ``inside''}} U_A \subseteq H^n$ for all $n$, and hence $\hat U \subseteq H$. By alternatingly following segments of $\alpha_1$ and $U_A$, we see that $\hat U$ connects $z_1$ and $z_2$ in $B(z,r_2) \cap H$.

Let $n \in \mathbb{N}$. By assumption, we can find a path $\alpha_3$ in $B(z,r_1) \cap H^n$ that connects $z_1$ to $z_2$. Since $\alpha_2 \cup \alpha_3 \subseteq B(z,r_1)$, the winding numbers of $\alpha_1 \cup \alpha_2$ and $\alpha_1 \cup \alpha_3$ around $U_A$ are the same. Therefore $U_A$ is disconnected from $\infty$ (and hence also from $\mathbb{R}$) by $\alpha_1 \cup \alpha_3$. Since $\alpha_1 \cup \alpha_3 \subseteq H^n$ and $H^n$ is simply connected, we must have $U_A \subseteq H^n$.

It remains to handle the case that (small neighbourhoods of) $z_1$ and $z_2$ lie in different components of $B(z,r_1) \setminus \alpha_1$. By construction, $\alpha_1 \cap B(z,r_1)$ consists of finitely many segments. Pick the segment $\tilde\alpha$ that bounds the component in which (a small neighbourhood of) $z_1$ lies, and let $\tilde z_2 \in \tilde\alpha$. Now $z_1$ and $\tilde z_2$ fulfil the conditions of the lemma again because any path in $B(z,r_1)$ from $z_1$ to $z_2$ needs to cross $\tilde\alpha$, and for each $n$ one such path lies in $H^n$ (by the assumption on $z_1,z_2$). Moreover, (small neighbourhoods of) $z_1$ and $\tilde z_2$ lie in the same component of $B(z,r_1) \setminus \alpha_1$. By the previous part of the proof, $z_1$ and $\tilde z_2$ are in the same connected component of $B(z,r_2) \cap H$. Repeating this argument, the lemma in the general case follows by induction.
\end{proof}

\begin{proposition}\label{thm:limits_of_simple_curves}
    Let $\gamma^n\colon [0,1] \to \barH$ be simple paths in $\H$ starting at $\gamma^n(0)=0$ and parametrised by half-plane capacity. Suppose $\|\gamma-\gamma^n\|_\infty \to 0$ and let $K_t = \fill(\gamma[0,t])$. Then the family $(K_t)_{t \in [0,1]}$ is parametrised by half-plane capacity and satisfies the local growth property.
\end{proposition}

\begin{proof}
\cref{thm:hcap_continuity} implies $\hcap(\gamma^n[0,t]) \to \hcap K_t$ for all $t$, so the parametrisation by half-plane is preserved.

To show that $(K_t)$ satisfies the local growth property, we will find for any $\varepsilon > 0$ some $\delta > 0$ such that for all $t$ there exists a crosscut of length less than $\varepsilon$ in $\H \setminus K_t$ that separates $K_{t+\delta} \setminus K_t$ from $\infty$. In the following, we call $H^n_t \defeq \H \setminus \gamma^n[0,t]$ and $H_t \defeq \H \setminus K_t$.

Since $\gamma$ is uniformly continuous, we find $\delta > 0$ such that $|\gamma(t)-\gamma(s)| < \varepsilon$ for all $|t-s| < \delta$.

Now let $t \in [0,1]$ and $z_1,z_2 \in K_{t+\delta} \setminus K_t$. It suffices to consider $z_1,z_2 \in \gamma[t,t+\delta] \setminus K_t$ since this set bounds $K_{t+\delta} \setminus K_t$. We claim that $z_1$ and $z_2$ are in the same connected component of $H_t \setminus \partial B(\gamma(t), 2\varepsilon)$. This will imply that there exists a segment of $\partial B(\gamma(t), 2\varepsilon) \cap H_t$ that separates $z_1$ and $z_2$ from $\infty$ in $H_t$ (This can be seen e.g. by mapping $H_t$ to $\H$). That segment is the desired crosscut.

By the choice of $\delta$ we have $z_1,z_2 \in H_t \cap B(\gamma(t),\varepsilon)$, and we can find $r>0$ such that $B(z_i,2r) \subseteq H_t \cap B(\gamma(t),\varepsilon)$, $i=1,2$. Let $n$ be large enough so that $\|\gamma-\gamma^n\|_\infty < r$. In particular, $B(z_i,r) \subseteq H^n_t \cap B(\gamma(t),\varepsilon)$, $i=1,2$.

Note that the uniform convergence of $\gamma^n$ implies $\gamma^n[0,t] \to K_t$ in the sense of kernel convergence. Since $\gamma^n$ are simple, $B(z_1,r)$ and $B(z_2,r)$ are connected by $\gamma^n(]t,t+\delta])$ in $H^n_t$. Moreover, $\gamma^n(]t,t+\delta]) \subseteq B(\gamma(t),\varepsilon+r)$ by the choice of $\delta$ and $n$. So by \cref{thm:ball_components_limit}, $B(z_1,r)$ and $B(z_2,r)$ are in the same connected component of $B(\gamma(t),2\varepsilon) \cap H_t$.
\end{proof}

We turn back to the question of characterising $\mathcal{S}$. Just from the definition of the support, for fixed $\kappa'\neq 8$,  we have a.s. $\gamma^{\kappa'}\in \mathcal{S}$.  Moreover, since piece-wise linear functions are in $\mathcal{D}$, any  $\gamma^\lambda$ that is approximated by a sequence of Loewner curves generated by piece-wise linear drivers is in $\mathcal{S}$. In particular, \cite[Theorem 2.2]{Tran} shows that if $\lambda$ is weakly $1/2$-H\"older and $|(\hat{f}^\lambda_t)'(iy)| \le Cy^{-\beta}$ for some $\beta < 1$ and all $t,y \in {]0,1]}$, then $\gamma^\lambda$ has such an approximation, hence is in $\mathcal{S}$.

We can also represent $\mathcal{S}$ by different sets of curves. For instance,
\begin{equation*}
    \mathcal{S} = \overline{\{\gamma^\lambda: \lambda \text{ piece-wise linear and } \lambda(0)=0\}}.
\end{equation*}
From the results in \cite{Tran} it follows that
\begin{equation*}
    \mathcal{S} = \overline{\{\gamma^\lambda: \lambda \text{ piece-wise square-root and } \lambda(0)=0\}}.
\end{equation*}
Or
\begin{align*}
    \mathcal{S} &= \overline{\{\gamma^\lambda \mid \lambda \in C^\infty \text{ and } \lambda(0)=0\}}\\
    &= \overline{\{\gamma \in C^\infty((0,1];\mathbb{H}) \mid \gamma \text{ simple, param. by half-plane capacity, and } \gamma(0)=0\}}.
\end{align*}
To see the last equality, suppose we have a simple curve $\gamma \in C^\infty((0,1];\mathbb{H})$ with $\gamma(0)=0$. Then we can approximate it by a simple smooth curve $\tilde\gamma \in C^\infty((0,1];\mathbb{H})$ with $\tilde\gamma(t) = i2\sqrt{t}$ on a very small time interval $t \in [0,\delta]$. Then $\tilde\gamma$ is driven by a smooth driving function (see \cite{EE}), so $\tilde\gamma\in\mathcal{S}$. (Strictly speaking, we also need to parametrise $\tilde\gamma$ by half-plane capacity, but this will not change the approximation much, as the proof of \cref{p:simple_curves} below shows.)

We can say more.
\begin{proposition}\label{p:simple_curves}
    Let $\gamma = C([0,1];\mathbb{H})$, with $\gamma(0)=0$, be a simple curve. Then $\gamma \in \mathcal{S}$ if and only if it is parametrised by half-plane capacity.
\end{proposition}

\begin{proof}
First let us show that any $\gamma \in \mathcal{S}$ is necessarily parametrised by half-plane capacity. Let $\gamma \in \mathcal{S}$ and find a Loewner curve $\gamma^\lambda$, e.g. a sample of SLE, (which by definition is parametrised by half-plane capacity) such that $\|\gamma-\gamma^\lambda\|_{\infty,[0,1]} < \varepsilon$. From \cref{thm:hcap_continuity} it follows that
\begin{equation*}
|\hcap(\gamma[0,t])-\hcap(\gamma^\lambda[0,t])| = |\hcap(\gamma[0,t])-2t| \leq \phi(\varepsilon) \quad \text{for all } t\in [0,1]
\end{equation*}
where $\phi:(0,\infty)\to (0,\infty)$ is a function depending on $\sup\{|\gamma(t)|:t\in [0,1]\}$ and satisfying $\phi(0+)=0$. Since this is true for all $\varepsilon > 0$, we have
\begin{equation*}
\hcap(\gamma[0,t]) = 2t \quad \text{for all } t\in [0,1]
\end{equation*}
as claimed.

    For the converse, let $\varepsilon > 0$. From \cite[Lemma 4.4]{Wer} it follows that there exists a linear interpolation $\gamma^{\mathcal P}$ that is simple and $\|\gamma-\gamma^{\mathcal P}\|_{\infty,[0,1]} < \varepsilon$. Then we can find a simple smooth curve $\eta \in C^\infty((0,1];\mathbb{H})$, with $\eta(0)=0$, such that $\|\eta-\gamma^{\mathcal P}\|_{\infty,[0,1]} < \varepsilon$, and consequently $\|\eta-\gamma\|_{\infty,[0,1]} < 2\varepsilon$. Let $\bar\eta$ be the reparametrisation of $\eta$ by half-plane capacity. Since $\bar\eta \in \mathcal{S}$, we only need to show that $\|\bar\eta-\gamma\|_{\infty,[0,1]}$ is small.

    Again, from \cref{thm:hcap_continuity}, since $\|\eta-\gamma\|_{\infty,[0,1]} < 2\varepsilon$, it follows that 
    \begin{equation*}
        |\hcap(\eta[0,t])-\hcap(\gamma[0,t])| \leq  \phi(\varepsilon) \quad \text{for all } t\in [0,1]
    \end{equation*}
    where $\phi:(0,\infty)\to (0,\infty)$ is a function depending on $\sup\{|\gamma(t)|:t\in [0,1]\}$ and satisfying $\phi(0+)=0$.
    
    Note that $\hcap(\gamma[0,t])=2t = \hcap(\bar\eta[0,t]])$ for all $t$. Hence, $\bar\eta(t) = \eta(s)$ where $|t-s| \le \phi(\varepsilon)/2$. That implies
    \begin{equation*}
        \|\bar\eta-\gamma\|_{\infty,[0,1]}\leq \tilde\phi(\varepsilon)
    \end{equation*}
    where $\tilde\phi:(0,\infty)\to (0,\infty)$ depending on $\phi$ and the uniform continuity of $\gamma$, and satisfying $\tilde\phi(0+)=0$.

\end{proof}

There are further questions that we have not answered. 
\begin{itemize}
    \item Can one strengthen the topology in Theorem \ref{t:main}? Note that the statement of \cref{t:main} is the same regardless of $\kappa$. But as shown in \cite{JVL}, for each $\kappa$ there exists an optimal $\alpha_*(\kappa)$ such that $\gamma^\kappa$ is $\alpha$-Hölder continuous for $\alpha < \alpha_*(\kappa)$. Ideally, we would like to characterise the support of \SLEk{} in the $\alpha$-Hölder space, or similarly, in the $p$-variation space where $p > p_*(\kappa)$ (see \cite{FT}). 
    
    (Note that it is proved in \cite{FS} that  $\gamma^\lambda$ is $1/2$-H\"older continuous on $[0,1]$ for $\lambda\in W^{1,2}$. Hence, almost surely the $\alpha$-H\"older norm of $(\gamma^\kappa-\gamma^\lambda)$ is finite for some $\alpha>0$.)
    
    \item We do not know how $ \P(\|\gamma^\kappa-\gamma^\lambda\|_{\infty,[0,1]}<\varepsilon)$ behaves as $\varepsilon\to 0^+$. However as $\kappa\to 0$, we believe that similarly to \cite{Wang} the following is true:
\begin{equation*}
    \lim_{\kappa\to 0}-\kappa\ln \P(\|\gamma^\kappa-\gamma^\lambda\|_{\infty,[0,1]}<\varepsilon) = \inf_{\{U\in W^{1,2}: \|\gamma^U-\gamma^\lambda\|_{\infty;[0,1]}<\varepsilon\}} \frac{1}{2} \int^1_0 |U'(t)|^2\,dt.
\end{equation*}
    \item To what extent does the converse of \cref{thm:limits_of_simple_curves} hold? Is every (not necessarily simple) curve, parametrised by half-plane capacity, that satisfies the local growth property  in fact in $\mathcal{S}$? (If not, is there a characterisation which curves are in $\mathcal{S}$?)\\
    This would give us a full characterisation of $\mathcal{S}$, generalising \cref{p:simple_curves} to general curves in $C([0,1];\barH)$.
    
    \emph{Update:} In an ongoing work, the second author gives a positive answer to this question.
\end{itemize}

\end{section}

\bibliographystyle{alpha}
\newcommand{\etalchar}[1]{$^{#1}$}

\end{document}